\documentclass[11pt,reqno]{amsart}
\usepackage{amsthm,amsfonts,amssymb,euscript}

\def\ddb#1{\sqrt{-1}\partial\bar{\partial}#1}
\def\dt#1{\frac{\partial}{\partial t}#1}

\newcommand{\dr}{\omega}

\newcommand{\e}{\epsilon}
\newcommand{\ka}{K\"{a}hler}

\theoremstyle{plain}
  \newtheorem{theorem}{Theorem}[section]
  \newtheorem{proposition}[theorem]{Proposition}
  \newtheorem{lemma}[theorem]{Lemma}
  \newtheorem{corollary}[theorem]{Corollary}

  \newtheorem{remark}[theorem]{Remark}

\theoremstyle{definition}

\numberwithin{equation}{section}
\begin{document}

\title[K\"{a}hler-Ricci flow]{The K\"{a}hler-Ricci flow on pseudoconvex domains}
\author{Huabin Ge}
\author{Aijin Lin}
\author{Liangming Shen}
\address{Department of Mathematics, Beijing Jiaotong University, Beijing, P.R. China, 100044}
\email{hbge@bjtu.edu.cn}
\address{College of Science, National University of Defense Technology, Changsha, Hunan, P.R.China, 410073}
\email{aijinlin@pku.edu.cn}
\address{Department of Mathematics, The University of British Columbia, Vancouver, B.C., Canada V6T 1Z2}
\email{lmshen@math.ubc.ca}

\begin{abstract}
 We establish the existence of \ka-Ricci flow on pseudoconvex domains with general initial metric without curvature bounds. Moreover we prove that this flow is simultaneously complete, and its normalized version converge to the complete \ka-Einstein metric, which generalizes Topping's works on surfaces.
\end{abstract}

\maketitle


\section{Introduction}\label{section1}

In \cite{GT1,GT2,Top1,Top2,Top3} Giesen and Topping began the study of the Ricci flow on surfaces with the initial metric which is incomplete or has unbounded curvature. Base on Hamilton's work on surfaces \cite{Ha}, they used different settings with respect to different conformal structures. To overcome the bad initial condition, they approximated the initial metrics and established the compactness of the approximating flow sequence using Perelman's pseudolocality theorem \cite{Pe}. In \cite{Top3} Topping pointed out that in high dimensional case, specifically for \ka\ manifolds, one possible approach is to study the complex Monge-Amp\'{e}re flow associated to the \ka-Ricci flow. In this paper we will consider the \ka-Ricci flow on pseudoconvex domains as a natural generalization.\\

As a fundamental study object in complex analysis and complex geometry, there are quite a lot researches on the pseudoconvex domains in different fields. For example, Fefferman considered the Bergman kernel and related it to the complex Monge-Amp\'{e}re equations \cite{Fe}. In \cite{CY} Cheng and Yau began to study the existence of \ka-Einstein metrics with negative Ricci curvature on general pseudoconvex manifolds. Furthermore Mok and Yau studied the completeness of the \ka-Einstein metrics on bounded domains of holomorphy and gave some curvature criterions related to those domains \cite{MY}.\\

Similar to compact \ka\ manifolds, the \ka-Ricci flow can be applied to the study of complete \ka\ metrics on noncompact \ka\ manifolds. In \cite{Chau}, Chau proved that the normalized \ka-Ricci flow converges to the complete \ka-Einstein metric provided that the initial metric $\dr_{0}$ is smooth, complete and satisfies that $Ric(\dr_{0})+\dr_{0}=\ddb f$ for some smooth and bounded function $f.$ This result also holds for pseudoconvex manifolds by \cite{CY,MY}. \\

In this paper, we will establish the existence and find the long time behavior of the \ka-Ricci flow with weaker initial metrics. Similar to Giesen and Topping, we can show the property of simultaneously completeness of the \ka-Ricci flow. First we consider a strictly pseudoconvex domain $\Omega$ with smooth boundary, i.e., there exists a smooth strictly plurisubharmonic function $\varphi$ defined on $\Omega$ such that $\partial\Omega=\{\varphi=0\}$ is smooth. By \cite{CY} there exists a complete metric defined by $-\ddb\log(-\varphi)$ which is asymptotically hyperbolic near the boundary. Suppose our initial metric is incomplete, we have the following result:

\begin{theorem}\label{mainthm1}
Given an incomplete \ka\ metric $\dr_{0}$ with uniformly bounded covariant derivatives on the strictly pseudoconvex domain $\Omega\subset\mathbb{C}^{n}$ with smooth boundary which is defined by $\partial\Omega=\{\varphi=0\},$ the solution to the \ka-Ricci flow
\begin{equation}\label{eq:krf1}
 \left\{ \begin{array}{rcl}
\dt\dr&=&-Ric(\dr)\\
\dr(0)&=&\dr_{0} \end{array}\right.
\end{equation}
exists on $\Omega\times[0,\infty)$ and the solution $\dr(t)$ is complete when $t>0.$ Moreover, the corresponding solution to the normalized \ka-Ricci flow converges to the complete \ka-Einstein metric as $t\to\infty.$
\end{theorem}
To establish this simultaneously complete flow we will transform the equation \eqref{eq:krf1} to the complex Monge-Amp\'{e}re flow. One natural problem is the behavior at the initial time. We will adapt similar approximating method in \cite{CLS} to overcome this difficulty. After that we will derive suitable a priori estimates to establish the existence and the limit behavior of this flow.

However this technique cannot be applied to general pseudoconvex domains or manifolds directly as we do not have such complete \ka\ metrics with the form $-\ddb\log(-\varphi).$ One possible idea is to construct an approximation sequence of solutions on exhausting strictly pseudoconvex domains with smooth
boundaries, as Cheng-Yau did for \ka-Einstein metric \cite{CY}. The main difficulty of this approach is that the \ka-Ricci flow solutions could not be locally uniformly controlled from above by the local Poincar\'{e} metric, which is quite different from the \ka-Einstein metric. This problem will result in the loss of the compactness of the solutions on exhausting domains. However, if there exists one complete background metric with strictly negative Ricci curvature and uniformly bounded curvature tensor, we can still derive a solution to the \ka-Ricci flow using a similar method to Theorem \ref{mainthm1}. In fact this background metric enables us to consider not only pseudoconvex manifolds. More precisely, we have
\begin{theorem}\label{mainthm2}
Given a \ka\ manifold $M$ with a noncomplete \ka\ metric $\dr_{0}$ with uniformly bounded covariant derivatives of all orders. If there exists a complete \ka\ metric $\dr_{M}$ such that $Ric(\dr_{M})\leq-C_{0}\dr_{M}$ for some constant $C_{0}>0$ and the curvature tensor of $\dr_{M}$ has uniformly bounded covariant derivatives, then there exists a long time solution $\dr(t)$ to the \ka-Ricci flow with the initial metric $\dr_{0}$ which is simultaneously complete as soon as $t>0.$ Furthermore, the normalized \ka-Ricci flow converges to the \ka-Einstein metric in Cheng-Yau \cite{CY} and the solution $\tilde{\dr}(t)$ is simultaneously complete as $t>0.$
\end{theorem}

This work could be thought as a higher dimensional generalization of Topping's simultaneously complete surface flow. However as we are not sure about the existence of nonconstant pluriharmonic functions on pseudoconvex domains, we cannot prove the uniqueness of the \ka-Ricci flow solution. We hope in the future we can establish some uniqueness results under some constraints. Moreover we hope to generalize our result to general pseudoconvex manifolds without the completeness of the background metric with negative Ricci curvature. Finally, we hope this work on pseudoconvex domains could help us to establish the existence of \ka-Ricci flow on more general noncompact \ka\ manifolds or Stein manifolds, which could be another approach to understanding the structures of such manifolds.\\

\noindent{\bf Acknowledgment.} The three authors would like to express their greatest respect to Professor Gang Tian for constant supports and encouragements. The third author also would like to thank Professor Albert Chau and Jingyi Chen for their interest in this work. Finally the third author would like to thank Man-Chun Lee for his advice on improving Theorem \ref{mainthm2}. The first author is partially supported by the NSFC Grant 11501027 and the second author is partially supported by the NSFC Grant 11401578.

\section{The parabolic Omori-Yau maximal principle and parabolic Schwarz Lemma}
In this section we will introduce basic tools in the analysis for complete metrics. One basic principle is the parabolic maximal principle in the settings of complete metrics. The elliptic version was used in \cite{CY,MY} while the parabolic version was used in \cite{Chau,CLS,LZ,Sh2}. Here we just cite two versions of maximal principles, which are not optimal but enough in our case:
\begin{proposition}\label{prop-max}
Suppose there exists a bounded curvature solution $\dr(t)$ to the \ka-Ricci flow on $M\times [0,T]$ such that $\dr(t)$ is equivalent to the initial complete metric $\dr_{0},$ then it holds that \\
(i) if a smooth function $\psi(x,t)$ is bounded from above on $M\times [0,T],$ then there exists a sequence of points $x_{k}$ and time $\bar{t}\in [0,T]$ such that $$\psi(x_{k},\bar{t})\to \sup\psi,\quad and\quad\ddb\psi(x_{k},\bar{t})\leq\frac{1}{k}\dr(x_{k},\bar{t});$$
(ii) if a smooth function $\psi(x,t)$ is bounded from above on $M\times [0,T],$ then there exists a sequence of points $x_{k}$ and time $\bar{t}\in [0,T]$ such that $$\psi(x_{k},\bar{t})\to \sup\psi,\quad |\nabla\psi(x,\bar{t})|_{\bar{t}}\leq\frac{1}{k}\quad and\quad(\dt-\Delta)\psi(x_{k},\bar{t})\geq-\frac{1}{k}.$$
Similarly it also holds for the bounds in the other direction.
\end{proposition}
Using the second maximal principle, we could derive the parabolic Schwarz Lemma:
\begin{proposition}\label{prop-schwarz}
Suppose there exists a bounded curvature solution $\dr(t)$ to the \ka-Ricci flow on $M\times [0,T]$ such that $\dr(t)$ is equivalent to the initial complete metric $\dr_{0},$ if there exists a \ka\ metric $\dr$ on $M$ such that its Ricci curvature is negative and bounded above by $-C,$ then it holds that $$\dr(t)^{n}\geq (Ct)^{n}\dr^{n}.$$
\end{proposition}
\begin{proof}
We adapt the proof in \cite{MY,Yau} to the parabolic case. As $\dr(t)$ solves the \ka-Ricci flow, we have
\begin{align*}
&(\dt-\Delta)\log\frac{\dr^{n}}{\dr(t)^{n}}\\=&R(t)-R(t)+tr_{\dr(t)}Ric(\dr)\leq-Cn(\frac{\dr^{n}}{\dr(t)^{n}})^{1/n}.
\end{align*}
Set $u:=\frac{\dr^{n}}{\dr(t)^{n}},$ it holds that $$(\dt-\Delta)u\leq-Cnu^{\frac{n+1}{n}},$$ thus we have that
$$(\dt-\Delta)((u+c)^{-1/n}-Ct)\geq-\frac{1}{n}(1+\frac{1}{n})\frac{|\nabla u|^{2}}{(u+c)^{2+1/n}}.$$
By Proposition \ref{prop-max} we conclude the proof.
\end{proof}
For future purposes, we introduce two important inequalities as following, which proof are quite standard, see for example, \cite{Chau,ST1}:
\begin{lemma}\label{lem-ineq}
Suppose there exists a solution $\dr(t)$ to the \ka-Ricci flow on $M\times [T',T],$ and there exist two \ka\ metrics $\dr_{1}$ with lower bisectional curvature bound $-k_{1}$ and $\dr_{2}$ with upper bisectional curvature bound $k_{2},$ then it holds that\\
(i)(Aubin-Yau)
\begin{equation}\label{eq:A-Y}
(\dt-\Delta)\log tr_{\dr_{1}}\dr(t)\leq k_{1}tr_{\dr(t)}\dr_{1};
\end{equation}
(ii)(Chern-Lu)
\begin{equation}\label{eq:C-L}
(\dt-\Delta)\log tr_{\dr(t)}\dr_{2}\leq k_{2}tr_{\dr(t)}\dr_{2}.
\end{equation}
\end{lemma}

\section{\ka-Ricci flow on strictly pseudoconvex domains with incomplete initial metric}

In this section we will prove Theorem \ref{mainthm1}. First we will recall the basic properties of complete \ka\ metrics on strictly pseudoconvex domain $\Omega\subset\mathbb{C}^{n}$ in \cite{CY}. Consider the metric $\bar{\dr}=-\ddb\log(-\varphi)$ where $\varphi$ is the smooth defining function of the domain $\Omega$ such that the boundary $\partial\Omega=\{\varphi=0\}.$ By \cite{CY} $\bar{\dr}$ is a complete metric with asymptotically bisectional curvature -1. Now consider an incomplete metric $\dr_{0},$ we want to verify that $\dr:=\dr_{0}+\bar{\dr}$ has the same asymptotical geometric behavior with $\bar{\dr}.$ We will adapt the local coordinates introduced in \cite{CY}. First, we have
$$g_{i\bar{j}}=g_{0i\bar{j}}-\frac{\varphi_{i\bar{j}}}{\varphi}+\frac{\varphi_{i}\varphi_{\bar{j}}}{\varphi^{2}},$$ which is obviously a complete metric and asymptotical to $\bar{\dr}$.
Now set the Hermitian metric $\tilde{g}_{i\bar{j}}=\varphi_{i\bar{j}}-\varphi g_{0i\bar{j}},$ we have
$$g^{i\bar{j}}=(-\varphi)\tilde{g}^{i\bar{l}}(\delta_{jl}-\frac{\tilde{g}^{k\bar{j}}\varphi_{k}\varphi_{\bar{l}}}{|\nabla\varphi|_{\tilde{g}}^{2}-\varphi})
=(-\varphi)(\tilde{g}^{i\bar{j}}-\frac{\varphi^{i}\varphi^{\bar{j}}}{|\nabla\varphi|_{\tilde{g}}^{2}-\varphi}),$$
where $\varphi^{i}=\tilde{g}^{i\bar{l}}\varphi_{\bar{l}}.$
By direct but sophisticated computations, we have
\begin{align*}
\Gamma_{ij}^{k}=&\frac{\varphi_{i}\delta_{jk}+\varphi_{j}\delta_{ik}}{-\varphi}
+\frac{\varphi_{ij}\varphi^{k}}{|\nabla\varphi|_{\tilde{g}}^{2}-\varphi}\\
&+(\tilde{g}^{k\bar{l}}-\frac{\varphi^{k}\varphi^{\bar{l}}}{|\nabla\varphi|_{\tilde{g}}^{2}-\varphi})
(-\varphi\partial_{i}g_{0j\bar{l}}-\varphi_{i}g_{0j\bar{l}}-g_{j}\delta_{0i\bar{l}}+\varphi_{j\bar{l}i}),\\
\end{align*}
and
\begin{align}\label{eq:cur}
R_{i\bar{j}k\bar{l}}=&-(g_{i\bar{j}}g_{k\bar{l}}+g_{i\bar{l}}g_{k\bar{j}})-\frac{1}{\varphi}\left(\tilde{g}_{p\bar{l}}R^{p}_{i\bar{j}k}(\varphi_{i\bar{j}})
-\frac{\varphi_{,ik}\varphi_{,\bar{j}\bar{l}}}{|\nabla\varphi|^{2}_{\tilde{g}}-\varphi}\right)
-g_{0k\bar{l},i\bar{j}}+(g_{0i\bar{j}}g_{0k\bar{l}}+g_{0i\bar{l}}g_{0k\bar{j}})\nonumber\\
&-\frac{1}{\varphi}(\varphi_{i\bar{j}}g_{0k\bar{l}}+\varphi_{i\bar{l}}g_{0k\bar{j}}+\varphi_{k\bar{l}}g_{0i\bar{j}}+\varphi_{k\bar{j}}g_{0i\bar{l}})
-\frac{1}{\varphi}\left(\varphi_{i}g_{0k\bar{l},\bar{j}}+\varphi_{\bar{j}}g_{0k\bar{l},i}
+\varphi_{k}g_{0i\bar{l},\bar{j}}+\varphi_{\bar{l}}g_{0k\bar{j},i}\right)\nonumber\\
&+g^{p\bar{q}}\left(g_{0p\bar{l},\bar{j}}g_{0k\bar{q},l}+g_{0p\bar{l},\bar{j}}(\frac{\varphi_{\bar{q}}\varphi_{,ik}}{\varphi^{2}}+\frac{\varphi_{i}}{\varphi}
g_{0k\bar{q}}+\frac{\varphi_{k}}{\varphi}g_{0i\bar{q}})+g_{0k\bar{q},i}(\frac{\varphi_{p}\varphi_{,\bar{j}\bar{l}}}{\varphi^{2}}
+\frac{\varphi_{\bar{l}}}{\varphi}g_{0p\bar{j}}+\frac{\varphi_{\bar{j}}}{\varphi}g_{0p\bar{l}})\right)\nonumber\\
&+\frac{g^{p\bar{q}}}{\varphi^{2}}\left(\varphi_{i}\varphi_{\bar{j}}g_{0k\bar{q}}g_{0p\bar{l}}+\varphi_{i}\varphi_{\bar{l}}g_{0k\bar{q}}g_{0p\bar{j}}
+\varphi_{k}\varphi_{\bar{j}}g_{0i\bar{q}}g_{0p\bar{l}}+\varphi_{k}\varphi_{\bar{l}}g_{0i\bar{q}}g_{0p\bar{j}}\right)\nonumber\\
&+\frac{g^{p\bar{q}}}{\varphi^{3}}\left(\varphi_{p}\varphi_{,\bar{j}\bar{l}}(g_{0i\bar{q}}\varphi_{k}+g_{0k\bar{q}}\varphi_{i})
+\varphi_{\bar{q}}\varphi_{,ik}(g_{0p\bar{j}}\varphi_{\bar{l}}+g_{0p\bar{l}}\varphi_{\bar{j}})\right),\nonumber\\
\end{align}
where all the covariant derivatives are with respect to the \ka\ metric $\ddb\varphi$ and $R^{p}_{i\bar{j}k}(\varphi_{i\bar{j}})$ also denotes the curvature tensor of $\ddb\varphi.$ Near the boundary of $\Omega$ the defining function $\varphi$ is smooth and derivatives of all orders are uniformly bounded, and the initial metric $g_{0}$ with its covariant derivatives all orders are also uniformly bounded. As $g^{p\bar{q}}=O(|\varphi|),$ it turns out that
$$R_{i\bar{j}k\bar{l}}=-(g_{i\bar{j}}g_{k\bar{l}}+g_{i\bar{l}}g_{k\bar{j}})+O(\frac{1}{\varphi^{2}}).$$ Moreover as $g_{i\bar{j}}=O(\frac{1}{\varphi^{2}}),$
the bisectional curvature of $\dr=\dr_{0}+\bar{\dr}$ is asymptotically $-1$ and all its covariant derivatives are asymptotically to 0, which are the same to the case in \cite{CY}. \\
This observation enables us to establish the short time existence of \ka-Ricci flow on pseudoconvex domains with initial metric $\dr=\dr_{0}+\bar{\dr}$ by Shi's existence theorems \cite{Sh1,Sh2}. In case that the initial metric $\dr_{0}$ is incomplete, we will derive uniform estimates for the initial metric $\dr_{\e}=\dr_{0}+\e\bar{\dr}$ and take the limit as $\e\to 0.$ To realize this approximation, we will transform the original \ka-Ricci flow \eqref{eq:krf1} to complex Monge-Amp\'{e}re flow equation as following.\\

Recall the computation in \cite{CY}, still under the same local coordinates, we have
$$Ric(\bar{\dr})+(n+1)\bar{\dr}=-\ddb\log(det(\varphi_{i\bar{j}})(|d\varphi|_{\varphi}^{2}-\varphi)),$$
where $|d\varphi|_{\varphi}$ is with respect to the \ka\ metric $\ddb\varphi.$ For simplicity we set
$$f:=\log(det(\varphi_{i\bar{j}})(|d\varphi|_{\varphi}^{2}-\varphi))$$ which is smooth with uniformly bounded derivatives of all orders.
Now write
\begin{equation}\label{eq:bg1}
\dr_{t}:=\dr_{0}+(n+1)t\bar{\dr},
\end{equation}
which is obviously positive for all time and equivalent to the complete metric $\bar{\dr}$ for all positive time. Assume $\dr(t)=\dr_{t}+\ddb u$ satisfies \ka-Ricci flow \eqref{eq:krf1}, it holds that
\begin{align*}
\ddb\dt u=&-(n+1)\dr-Ric(\dr(t))=Ric(\bar{\dr})-Ric(\dr(t))+\ddb f\\
=&\ddb\log\frac{(\dr_{t}+\ddb u)^{n}}{\bar{\dr}^{n}}+\ddb f,
\end{align*}
then we could derive a complex Monge-Amp\'{e}re flow equation
\begin{equation}\label{eq:cma1}
\left\{ \begin{array}{rcl}
\dt u&=&\log\frac{(\dr_{t}+\ddb u)^{n}}{\bar{\dr}^{n}}+f\\
u(0)&=&0 \end{array}\right.
\end{equation}
whose solution will give rise to a solution to \ka-Ricci flow \eqref{eq:krf1}, although the converse may not be true. The main problem of equation \eqref{eq:cma1} is the degeneracy of $\dr_{t}$ at $t=0.$ To overcome this point we adapt the approximation procedure in \cite{CLS} which deals with the similar
problem in the situation of cusp type metrics. \\

Set $\dr_{t,\e}=\dr_{0}+(n+1)(t+\e)\bar{\dr}$ and $\dr_{\e}(t)=\dr_{t,\e}+\ddb u_{\e},$ we consider the following approximating Monge-Amp\'{e}re flow equation:
\begin{equation}\label{eq:cma1-app}
\left\{ \begin{array}{rcl}
\dt u_{\e}&=&\log\frac{(\dr_{t,\e}+\ddb u_{\e})^{n}}{\bar{\dr}^{n}}+f\\
u_{\e}(0)&=&0 \end{array}\right.
\end{equation}
By the arguments above for each $\e>0$ we know that \eqref{eq:cma1-app} has a complete solution with bounded curvature on $[0,T_{max})$. Moreover by existence results of corresponding Monge-Amp\'{e}re flow equation (see e.g., Lemma 4.1 of \cite{Chau} or Theorem 4.1 of \cite{LZ}) we have smooth and bounded potential solution $u_{\e}$ depending on $t$ and $\e$ as long as $t<T_{max}.$ We will apply Omori-Yau maximal principle in Proposition
\ref{prop-max} to derive uniformly a priori estimates and establish the existence of long time solution. First, we have
\begin{lemma}\label{lem-C0}
There exists a constant $C(t)>0$ only depending on $t$ and bounded on $[0,T_{\max})$ such that
\begin{equation}\label{eq:C0}
|u_{\e}(\cdot,t)|\leq C(t).
\end{equation}
\end{lemma}
\begin{proof}
By the argument above we only need to derive a priori estimates. We take $\psi:=u_{\e}-Ct$ for some undetermined constant $C$ as \cite{CLS}. For $T<T_{max}$ fix $\bar{t}\in [0,T]$ such that $\max\psi(\cdot,\bar{t})$ attains the maximum of $\psi$ on $\Omega\times[0,T]$ and without loss of generality we can assume $\bar{t}>0.$ If there exists a point $p\in\Omega$ such that $\psi{\e}(p,\bar{t})$ attains this maximum, apply classical maximal principle to \eqref{eq:cma1-app} at this point, it follows that
\begin{equation}\label{eq:psi}
\dt\psi\leq\log\frac{\dr_{t,\e}^{n}}{\bar{\dr}^{n}}+f-C\leq-1
\end{equation}
for some sufficiently large $C$ which is independent of $\e.$ However this contradicts with the maximality of $\psi{\e}(p,\bar{t})$ in the assumption unless $\bar{t}=0.$ Thus we conclude that $$\sup u_{\e}(\cdot,t)\leq Ct.$$ In general case, since we have bounded curvature solution as the argument before, by Omori-Yau maximal principle in Proposition \ref{prop-max} there exist a sequence of points $p_{k}\in\Omega$ such that $\psi(p_{k},\bar{t})\to\sup\psi$ on $\Omega\times[0,T],$ and $\ddb\psi(p_{k},\bar{t})\leq\frac{\dr_{\bar{t},\e}}{k}.$ Apply this to \eqref{eq:cma1-app} again we
conclude that $\dt\psi(p_{k},\bar{t})\leq-1$ for constant $C$ independent of $\e$ and all sufficiently large $k.$ On the other hand by the existence of bounded curvature solution we conclude that $\partial^{2}_{t}\psi(p_{k},\bar{t})$ are uniformly bounded independent of $k.$ These facts contradict the assumption that
$\psi(p_{k},\bar{t})\to\sup\psi$ unless $\bar{t}=0$ and we conclude the upper bound for $u_{\e}(\cdot,t)$ as well.

Using almost the same argument we could derive the uniform lower bound for $u_{\e}(\cdot,t)$ and conclude this lemma.
\end{proof}
We will apply Omori-Yau maximal principle several times in this work. For simplicity we will always assume the maximum could be attained otherwise just apply similar argument to Lemma \ref{lem-C0}. Now we will give the uniform estimates for the time derivative $\dot{u}_{\e}:$
\begin{lemma}\label{lem-time dev}
There exists two bounded constants $C_{1}(t),C_{2}(t)>0$ only depending on $C(t)$ in Lemma \ref{lem-C0} such that
\begin{equation}\label{eq:time-dev}
n\log t-C_{1}(t)\leq\dot{u}_{\e}(\cdot,t)\leq\frac{C_{2}(t)}{t}+n.
\end{equation}
\end{lemma}
\begin{proof}
We will adapt the inequalities in \cite{CLS,ST2} to this lemma. For the upper bound, we have the following inequality
\begin{align*}
(\dt-\Delta_{\e})(t\dot{u}_{\e}-u_{\e}-nt)=&t(\dt-\Delta_{\e})\dot{u}_{\e}+\Delta_{\e}u_{\e}-n\\
=&t(n+1)tr_{\dr_{\e}}\bar{\dr}-tr_{\dr_{\e}}\dr_{t}=-tr_{\dr}\dr_{0,\e}<0,
\end{align*}
where $\Delta_{\e}$ is the Laplacian with respect to $\dr_{\e}(t).$
Apply Omori-Yau maximal principle combined with Lemma \ref{lem-C0}, we conclude the uniformly upper bound estimate.
For the lower bound, consider the following inequality
\begin{align*}
(\dt-\Delta_{\e})(\dot{u}_{\e}-n\log t)=&(n+1)tr_{\dr_{\e}}\bar{\dr}-\frac{n}{t}\\
\geq&C(\frac{\bar{\dr}^{n}}{\dr_{\e}^{n}})^{1/n}-\frac{n}{t}.
\end{align*}
Observe that $\dot{u}_{\e}-n\log t$ tends to infinity uniformly at $t=0,$ apply Omori-Yau maximal principle again and
still assume the minimum could be attained (otherwise choose a minimizing sequence) at $(p,\bar{t})$ for some $\bar{t}>0,$ it holds that
$$\dot{u}_{\e}(p,\bar{t})=\log\frac{\dr_{\e}^{n}}{\bar{\dr}^{n}}(p,\bar{t})-f(p)\geq n\log\bar{t}-C',$$
thus it holds that $\dot{u}_{\e}\geq n\log t-C_{1}(t)$ which is independent of $\e.$
\end{proof}
Base on the lemmas above we could derive a Laplacian estimate:
\begin{lemma}\label{lem-C2}
There exist two constants $C_{3}(t),C_{4}(t)>0$ independent of $\e$ such that
\begin{equation}\label{eq:C2}
C_{3}(t)\bar{\dr}\leq\dr_{\e}(t)\leq C_{4}(t)\bar{\dr}.
\end{equation}
\end{lemma}
\begin{proof}
By \eqref{eq:cur} or the curvature formula in \cite{CY} the bisectional curvature of $\bar{\dr}$ is bounded from above by a constant $k\geq 0.$
Then apply Chern-Lu Inequality \eqref{eq:C-L}, we have
\[(\dt-\Delta_{\e})(\log tr_{\dr_{\e}}\bar{\dr}-A\dot{u}_{\e}+B\log t)\leq(k-A(n+1))tr_{\dr_{\e}}\bar{\dr}+\frac{B}{t}.\]
First we choose $A$ such that $k-A(n+1)=-1,$ then consider that at $t=0,$ it holds that
$$tr_{\dr_{\e}}\bar{\dr}\leq\frac{1}{\e},\;\dot{u}_{\e}\geq n\log(n+1)\e,$$ thus
$\log tr_{\dr_{\e}}\bar{\dr}-A\dot{u}_{\e}+B\log t$ tends to $-\infty$ and the maximum could only be attained for $t>0.$ As we argued before we
still assume this maximum is attained at $(p,\bar{t})$ then it holds that $tr_{\dr_{\e}}\bar{\dr}(p,\bar{t})\leq\frac{B}{\bar{t}},$
which implies $$\log tr_{\dr_{\e}}\bar{\dr}-A\dot{u}_{\e}+B\log t\leq(B-1-n)\log t+AC_{1}(t)+\log B<\infty$$ as long as $B\geq n+1.$ Thus choose
$B\geq n+1,$ $\log tr_{\dr_{\e}}\bar{\dr}-A\dot{u}_{\e}+B\log t$ is uniformly bounded from above which implies
$$tr_{\dr_{\e}}\bar{\dr}\leq e^{\frac{C}{t}}$$ by Lemma \ref{lem-time dev}. Next consider the equation \eqref{eq:cma1-app} we have
$$\dr_{\e}^{n}=\bar{\dr}^{n}e^{\dot{u}_{\e}-f}\leq\bar{\dr}^{n}e^{\frac{C_{2}(t)}{t}+C'},$$ which concludes the proof.
\end{proof}
As \cite{Chau,ST2} we can proceed to derive higher order estimates as following:
\begin{lemma}\label{lem-high}
There exists a constant $C_{5}(t)$ such that $$|\nabla_{\bar{\dr}}\dr_{\e}|^{2}_{\dr_{\e}}\leq e^{\frac{C_{5}(t)}{t}},$$ and a constant $C_{t_{1},t_{2},k}>0$
for $0<t_{1}<t_{2}<T_{max}$ such that
\begin{equation}\label{eq:high order}
|u_{\e}|_{C^{k}([t_{1},t_{2}]\times\Omega,\bar{\dr})}\leq C_{t_{1},t_{2},k}.
\end{equation}
\end{lemma}
Given the Lemmas above, by the compactness theorem, we conclude that there exists a sequence $\e_{k}\searrow 0$ such that $u_{\e_{k}}$ converges to a solution
$u$ satisfying \eqref{eq:cma1} in the space of $C^{0}([0,T_{max})\times\Omega)\cap C^{\infty}((0,T_{max})\times\Omega).$ Moreover by the Laplacian estimate in \ref{lem-C2} the corresponding flow solution will be complete as long as $t>0.$ By use of maximal principle again
we establish the uniqueness of the solution to \eqref{eq:cma1}:
\begin{proposition}\label{prop-unique}
There exists a unique solution $u$ satisfying \eqref{eq:cma1} in the space of $C^{0}([0,T_{max})\times\Omega)\cap C^{\infty}((0,T_{max})\times\Omega).$
\end{proposition}
\begin{proof}
We only need to prove the uniqueness. Suppose we have two solutions $u_{1},u_{2}$ satisfying \eqref{eq:cma1}, then write $v=u_{1}-u_{2}$ and $\dr_{2}=\dr_{t}+\ddb u_{1}=\dr_{0}+(n+1)t\bar{\dr}+\ddb u_{2},$ we have the following equation for $v:$
\begin{equation*}
\left\{ \begin{array}{rcl}
\dt v&=&\log\frac{(\dr_{2}+\ddb v)^{n}}{\dr_{2}^{n}}\\
v(0)&=&0 \end{array}\right.
\end{equation*}
$v$ is also a bounded function. If the maximum of $v(t)$ is attained for each $t$ obviously we have $v\leq 0.$ For the general case, we apply Omori-Yau maximal principle again that for each $t>0$ there exists a sequence of points $p_{k}$ such that $v(p_{k},t)\to\sup v(\cdot,t)$ and $\ddb v(p_{k},t)\leq\frac{1}{k}\dr_{2}$ as $\dr_{2}(t)$ is a bounded curvature metric for $t>0.$ Thus it holds that
$$\dt\sup v(\cdot,t)\leq\lim\sup_{k}\log(1+\frac{1}{k})^{n}=0$$ in the sense of distribution, which implies $u_{1}\leq u_{2}.$ Similarly we have
$u_{1}\geq u_{2}$ and the uniqueness is approved.
\end{proof}
Consider Lemma \ref{lem-C0} again, we have a more precise estimate for $u_{\e}$ essentially by maximal principle:
\begin{equation}\label{eq:C0-precise}
\int_{0}^{t}n\log(n+1)(s+\e)ds+t\inf_{\Omega}f\leq u_{\e}(t)\leq\int_{0}^{t}n\log(c+(n+1)(s+\e))ds+t\sup_{\Omega}f,
\end{equation}
where $c$ is a constant such that $\dr_{0}\leq c\bar{\dr}$ on $\Omega.$ Note that this bound is uniformly bounded on any bounded time interval
$[0,T]$ which essentially implies the uniformly bounded high order estimates for $u_{\e}$ until $t=T_{max}.$ This essentially implies that
derivatives of $u$ of all orders at $t=T_{max}$ is uniformly bounded. By the short time existence, this flow could be extended to another time interval, and by \eqref{eq:C0-precise} again, the solution could be extended to infinite time:
\begin{theorem}\label{thm-long time}
There exists a unique solution $u$ to the equation \eqref{eq:cma1} in the space of $C^{0}([0,+\infty)\times\Omega)\cap C^{\infty}((0,+\infty)\times\Omega).$
This solution gives rise to a family of \ka\ metrics $\dr(t)$ which satisfy the \ka-Ricci flow \eqref{eq:krf1} on $[0,+\infty)\times\Omega$ and are complete as long as $t>0.$
\end{theorem}
\begin{remark}
Theorem \ref{thm-long time} finishes the proof the first part of Theorem \ref{mainthm1}. However we note that this only generates a solution to the \ka-Ricci flow \eqref{eq:krf1} from the unique solution to the complex Monge-Amp\'{e}re flow \eqref{eq:cma1} but the \ka-Ricci flow solution itself may not be unique even in the class of simultaneously complete flow solutions.
\end{remark}
Now we begin the proof the second part of Theorem \ref{mainthm1}. We consider the following equation of normalized \ka-Ricci flow:
\begin{equation}\label{eq:krf2}
 \left\{ \begin{array}{rcl}
\dt\tilde{\dr}&=&-Ric(\tilde{\dr})-(n+1)\tilde{\dr}\\
\tilde{\dr}(0)&=&\tilde{\dr}_{0} \end{array}\right.
\end{equation}
Namely the solution of \eqref{eq:krf2} could be derived from the unnormalized \ka-Ricci flow \eqref{eq:krf1} by
\begin{equation}\label{eq:normal}
\tilde{\dr}(t)=e^{-(n+1)t}\dr(\frac{e^{(n+1)t}-1}{n+1}),
\end{equation}
thus the long time existence of the solution to \eqref{eq:krf1} in Theorem \ref{thm-long time} also work for \eqref{eq:krf2}.
However to prove the convergence result, we still need to analyse the behavior of the solution to the corresponding complex Monge-Amp\'{e}re flow equation.
Similar to \cite{ST1,TZ} we set $$\tilde{\dr}_{t}=e^{-(n+1)t}\dr_{0}+(1-e^{-(n+1)t})\bar{\dr},$$ and $\tilde{\dr}(t)=\tilde{\dr}_{t}+\ddb\tilde{u},$ thus
we have the following complex Monge-Amp\'{e}re flow equation:
\begin{equation}\label{eq:cma2}
\left\{ \begin{array}{rcl}
\dt\tilde{u}&=&\log\frac{(\tilde{\dr}_{t}+\ddb\tilde{u})^{n}}{\bar{\dr}^{n}}-(n+1)\tilde{u}+f\\
\tilde{u}(0)&=&0 \end{array}\right.
\end{equation}
By the relation \eqref{eq:normal} we have a unique long time solution $\tilde{u}(t)$ to \eqref{eq:cma2} in the space of $C^{0}([0,+\infty)\times\Omega)\cap C^{\infty}((0,+\infty)\times\Omega).$ The remaining issue is to investigate the limit behavior of the solution $\tilde{u}(t)$ as long as $t$ tends to infinity.
Still by Omori-Yau maximal principle, we have the following $C^{0}$-estimate for $\tilde{u}:$
\begin{lemma}\label{lem-C0-nor}There exists a uniform constant $C>0$ such that
\begin{equation}\label{eq:C0-nor}
|\tilde{u}(t)|\leq C
\end{equation}
on $[0,+\infty)\times\Omega.$
\end{lemma}
\begin{proof}
By Lemma \ref{lem-C0} and the relation \eqref{eq:normal} we know that for each $t>0$ the solution $\tilde{u}(t)$ is bounded by a constant depending on $t.$
For each $t>0$ we still assume the maximum and minimum can be attained at some points otherwise we use the argument of limiting sequence. Similar to \eqref{eq:C0-precise} we have the following inequality with the same constant $c:$
\begin{align*}
&-(1+e^{-(n+1)t})|f|_{C^{0}}+e^{-(n+1)t}\int_{0}^{t}ne^{(n+1)s}\log(1-e^{-(n+1)s})ds\leq\tilde{u}(t)\\ \leq
&(1+e^{-(n+1)t})|f|_{C^{0}}+e^{-(n+1)t}\int_{0}^{t}ne^{(n+1)s}\log(ce^{-(n+1)s}+(1-e^{-(n+1)s}))ds,
\end{align*}
thus the uniform estimate for $\tilde{u}(t)$ follows.
\end{proof}
Now we begin to establish the estimate for $\dot{\tilde{u}}:$
\begin{lemma}\label{lem-time dev-nor}
There exist two uniform constants $C_{1},C_{2}>0$ and $t_{0}>0$ such that for all $t>t_{0}$,
\begin{equation}\label{eq:time-dev-nor}
-C_{1}\leq\dot{\tilde{u}}(t)\leq C_{2}te^{-(n+1)t}.
\end{equation}
\end{lemma}
\begin{proof}
First by \eqref{eq:cma2}, we have the evolution equation for $\dot{\tilde{u}}:$
\begin{equation}\label{eq:heat-time dev}
(\dt-\Delta)\dot{\tilde{u}}=tr_{\tilde{\dr}}\dot{\tilde{\dr}}_{t}-(n+1)\dot{\tilde{u}}=
(n+1)e^{-(n+1)t}tr_{\tilde{\dr}}(\bar{\dr}-\dr_{0})-(n+1)\dot{\tilde{u}}.
\end{equation}
Then we have the following inequality:
\begin{align*}
&(\dt-\Delta)((e^{(n+1)t}-1)\dot{\tilde{u}}-(n+1)\tilde{u}-n(n+1)t)\\
=&(n+1)\left((1-e^{-(n+1)t})tr_{\tilde{\dr}}(\bar{\dr}-\dr_{0})-(e^{(n+1)t}-1)\dot{\tilde{u}}+e^{(n+1)t}\dot{\tilde{u}}\right)-(n+1)\dot{\tilde{u}}\\
&+(n+1)(n-tr_{\tilde{\dr}}\tilde{\dr}_{t})-n(n+1)\\
=&-(n+1)tr_{\tilde{\dr}}\dr_{0}\leq 0.
\end{align*}
As the solution is smooth and with uniformly bounded derivatives for some positive time $t$ then combine Lemma \ref{lem-C0-nor} and Omori-Yau maximal principle we conclude the upper bound for $\dot{\tilde{u}}.$ On the other hand apply \eqref{eq:heat-time dev} again, for a large positive $T$ we have
\begin{align*}
&(\dt-\Delta)((1-e^{(n+1)(t-T)})\dot{\tilde{u}}+(n+1)\tilde{u})\\
=&(n+1)\left((1-e^{(n+1)(t-T)})e^{-(n+1)t}tr_{\tilde{\dr}}(\bar{\dr}-\dr_{0})-n+tr_{\tilde{\dr}}\tilde{\dr}_{t}\right)\\
=&(n+1)tr_{\tilde{\dr}}\tilde{\dr}_{T}-n(n+1).
\end{align*}
By the argument of the upper bound of $\dot{\tilde{u}}$ if $t\geq T$ it holds that $(1-e^{(n+1)(t-T)})\dot{\tilde{u}}+(n+1)\tilde{u}$ is uniformly bounded from below. Now assume that its minimum could be attained at some $t'<T$ and some point $p.$ Then by Omori-Yau maximal principle at $(p,t')$ it holds that $tr_{\tilde{\dr}}\tilde{\dr}_{T}\leq n,$ which implies that $\tilde{\dr}\geq C\tilde{\dr}_{T}.$ It follows that
\begin{align*}
&(1-e^{(n+1)(t-T)})\dot{\tilde{u}}(t)+(n+1)\tilde{u}(t)\\ \geq&(1-e^{(n+1)(t'-T)})\dot{\tilde{u}}(p,t')+(n+1)\tilde{u}(p,t')\\
=&(1-e^{(n+1)(t'-T)})(\log\frac{\tilde{\dr}^{n}}{\bar{\dr}^{n}}+f-(n+1)\tilde{u})(p,t')+(n+1)\tilde{u}(p,t')\\
\geq&(1-e^{(n+1)(t'-T)})\log\frac{\tilde{\dr}_{T}^{n}}{\bar{\dr}^{n}}+C\geq C,
\end{align*}
when $t\in [0,T].$ As this normalized flow exists for infinite time, we can choose $T=\infty,$ thus the lower bound for $\dot{\tilde{u}}$ follows.
\end{proof}
The Laplacian estimate could be derived from a slight modification of Chern-Lu Inequality:
\begin{lemma}\label{lem-C2-nor}
There exists a uniform constant $C_{3}>1$ such that for any $t>t_{0}>0,$
\begin{equation}\label{eq:C2-nor}
C_{3}^{-1}\bar{\dr}\leq\tilde{\dr}(t)\leq C_{3}\bar{\dr}.
\end{equation}
\end{lemma}
\begin{proof}
For \eqref{eq:C-L} in Proposition \ref{lem-ineq}, we have the following modification version in in the settings of normalized \ka-Ricci flow with same conditions:
\begin{equation}\label{eq:C-L-nor}
(\dt-\Delta)\log tr_{\tilde{\dr}(t)}\dr_{2}\leq k_{2}tr_{\tilde{\dr}(t)}\dr_{2}-n-1.
\end{equation}
Now we have the following inequality:
\begin{align*}
&(\dt-\Delta)(\log tr_{\tilde{\dr}}\bar{\dr}-\frac{k+2}{n+1}(\dot{\tilde{u}}+(n+1)\tilde{u}))\\
\leq&k tr_{\tilde{\dr}}\bar{\dr}-n-1-(k+2)(tr_{\tilde{\dr}}\bar{\dr}-n)\leq -2 tr_{\tilde{\dr}}\bar{\dr}+nk+n-1,
\end{align*}
where $k$ is a positive upper bound of the bisectional curvature of $\bar{\dr}.$
Still by Omori-Yau maximal principle and assuming that the maximum could be attained at some point $(p,t')$ for $t>t_{0},$
it holds that $tr_{\tilde{\dr}}\bar{\dr}(p,t')\leq\frac{nk+n-1}{2},$ then it follows that
$$\log tr_{\tilde{\dr}}\bar{\dr}-\frac{k+2}{n+1}(\dot{\tilde{u}}+(n+1)\tilde{u})\leq(\log tr_{\tilde{\dr}}\bar{\dr}-\frac{k+2}{n+1}(\dot{\tilde{u}}+(n+1)\tilde{u}))(p,t')\leq C,$$
by the above two lemmas and then $\bar{\dr}\leq C\tilde{\dr}(t).$ For the other direction, note that
from \eqref{eq:cma2} we have $$\log\frac{\tilde{\dr}^{n}}{\bar{\dr}^{n}}=\dot{\tilde{u}}+(n+1)\tilde{u}-f\leq C,$$
thus the other side bound and consequently this lemma follow.
\end{proof}
By standard procedures we can establish the uniform high order estimates as $t$ tends to infinity. As it is shown that $\dot{\tilde{u}}\leq C_{2}te^{-(n+1)t}$
we find that $\tilde{u}$ is almost decreasing as $t$ tends to infinity. Considering the uniform boundness of $\tilde{u}$ by Lemma \ref{lem-C0-nor}, $\tilde{u}(t)$ converges to a unique bounded limit function $\tilde{u}_{\infty}$ in $C^{0}$ norm. Moreover this convergence is essentially smooth by the uniform high order estimates. Thus the right hand side of the equation
$$\dt\tilde{u}=\log\frac{(\tilde{\dr}_{t}+\ddb\tilde{u})^{n}}{\bar{\dr}^{n}}-(n+1)\tilde{u}+f$$
uniformly converges to a smooth function on $\Omega$ which forces $\dot{\tilde{u}}$ converges to $0$ as $\tilde{u}(t)$ converges. It follows that
the limit function $\tilde{u}_{\infty}$ satisfies the complex Monge-Amp\'{e}re equation:
\begin{equation}\label{eq:limit cma}
\log\frac{(\bar{\dr}+\ddb\tilde{u}_{\infty})^{n}}{\bar{\dr}^{n}}=(n+1)\tilde{u}_{\infty}-f,
\end{equation}
 which implies that
$$Ric(\bar{\dr}+\ddb\tilde{u}_{\infty})=-(n+1)(\bar{\dr}+\ddb\tilde{u}_{\infty}),$$
thus we have the following proposition:
\begin{proposition}\label{prop-nor-converge}
There exist a unique family of solutions $\tilde{u}(t)$ to the normalized complex Monge-Amp\'{e}re flow \eqref{eq:cma2} and they converge smoothly to a
function $\tilde{u}_{\infty}$ which solves the complex Monge-Amp\'{e}re equation \eqref{eq:limit cma}. Consequently there exist a family of solutions $\tilde{\dr}(t)$ (which may not be unique) to the normalized \ka-Ricci flow \eqref{eq:krf2} on $[0,+\infty)\times\Omega.$ Moreover this family of solutions are simultaneously complete and converge to the complete \ka-Einstein metric $\tilde{\dr}_{\infty}$ in $C^{\infty}$ sense.
\end{proposition}
Thus we conclude the limit behavior of the normalized simultaneously complete \ka-Ricci flow on smooth pseudoconvex domains and the proof of Theorem \ref{mainthm1} is complete.

\section{\ka-Ricci flow on general manifolds with a complete background metric}

In this section we will complete the proof of Theorem \ref{mainthm2}. In this situation we will consider $\dr_{t}:=\dr_{0}-tRic(\dr_{M})$ as the new evolving
background metric. First we still need to verify that this background metric has uniformly bounded bisectional curvature for each $t>0.$ Similar to \eqref{eq:cur}, we have the following bisectional curvature expansion of $\bar{g}:=g_{0}+h$ where $g_{0}$ and $h$ are the metric tensor of $\dr_{0}$ and $-tRic(\dr_{M}):$
$$
\bar{R}_{i\bar{j}k\bar{l}}=-\partial_{k\bar{l}}g_{0i\bar{j}}-\partial_{k\bar{l}}h_{i\bar{j}}\\
+\bar{g}^{p\bar{q}}(\partial_{k}g_{0i\bar{q}}+\partial_{k}h_{i\bar{q}})(\partial_{\bar{l}}g_{0p\bar{j}}+\partial_{\bar{l}}h_{p\bar{j}}).
$$
By the uniformly bounded properties of the covariant derivatives of $\dr_{0}$ and $\dr_{M},$ it is implied that the bisectional curvature of $\bar{g}$
is uniformly bounded. This enables us to reuse the approximation method in the proof of Theorem \ref{mainthm1}. First set $\dr(t)=\dr_{t}+\ddb u$ as the flow solution then the \ka-Ricci flow equation can be transformed to the following complex Monge-Amp\'{e}re flow equation:
\begin{equation}\label{eq:cma-general}
\left\{ \begin{array}{rcl}
\dt u&=&\log\frac{(\dr_{t}+\ddb u)^{n}}{\dr_{M}^{n}}\\
u(0)&=&0 \end{array}\right.
\end{equation}
To overcome the difficulty at $t=0,$ similarly, we set $\dr_{t,\e}:=\dr_{0}-(t+\e)Ric(\dr_{M})$ and consider the following approximating flow equation:
\begin{equation}\label{eq:cma-general-app}
\left\{ \begin{array}{rcl}
\dt u_{\e}&=&\log\frac{(\dr_{t,\e}+\ddb u_{\e})^{n}}{\dr_{M}^{n}}\\
u_{\e}(0)&=&0 \end{array}\right.
\end{equation}
By the argument above for any closed time interval $[0,T]$ the evolving background metric $\dr_{t,\e}$ has uniformly bounded bisectional curvature, which enables us to apply the Omori-Yau maximal principle again to derive a priori estimates for $u_{\e}.$ As $C_{1}\dr_{M}\leq -Ric(\dr_{M})\leq C_{2}\dr_{M}$
for two positive constants $C_{1},C_{2}$ by the assumptions, it follows that $$C_{1}(t+\e)\dr_{M}\leq \dr_{t,\e}\leq (c'+C_{2}(t+\e))\dr_{M}$$ where $c'>0$ is
the constant such that $\dr_{0}\leq c'\dr_{M}$ as $\dr_{M}$ is complete. Similar to Lemma \ref{lem-C0} by applying the Omori-Yau maximal principle we have the following estimate:
\begin{equation}\label{eq:general C0}
n\int_{0}^{t}\log C_{1}(s+\e)ds\leq u_{\e}\leq n\int_{0}^{t}\log(c'+C_{2}(s+\e))ds
\end{equation}
This estimate also implies that $|u_{\e}(t)|\leq C(t)$ which is independent of $\e.$ Next we need to derive a uniform estimate for $\dot{u}_{\e}.$ Similar to
Lemma \ref{lem-time dev}, we have the following lemma:
\begin{lemma}\label{lem-general time dev}
There exist constants $C_{1}(t),C_{2}(t)>0$ which only depend on $t$ such that
\begin{equation}\label{eq:general time dev}
n\log t-C_{1}(t)\leq\dot{u}_{\e}(t)\leq\frac{C_{2}(t)}{t}+n.
\end{equation}
\end{lemma}
The proof is almost the same to Lemma \ref{lem-time dev} and we skip it. Next we will derive a uniform Laplacian estimate for $u_{\e}.$ As the background metric
$\dr_{M}$ is assumed to have uniformly bounded bisectional curvature, similar to Lemma \ref{lem-C2}--\ref{lem-high}
 we have the following high order derivative estimates:
\begin{lemma}\label{lem-general-high}
There exist constants $C_{3}(t),C_{4}(t),C_{5}(t)>0$ independent of $\e$ such that
\begin{align}\label{eq:general-C2}
C_{3}(t)\dr_{M}\leq\dr_{\e}(t)&\leq C_{4}(t)\dr_{M},\\
|\nabla_{\dr_{M}}\dr_{\e}|^{2}_{\dr_{\e}}&\leq e^{\frac{C_{5}(t)}{t}},
\end{align}
and a constant $C_{t_{1},t_{2},k}>0$
for $0<t_{1}<t_{2}<T_{max}$ such that
\begin{equation}\label{eq:high order}
|u_{\e}|_{C^{k}([t_{1},t_{2}]\times M,\dr_{M})}\leq C_{t_{1},t_{2},k}.
\end{equation}
\end{lemma}
Thus similarly by the uniform estimates above we derive that there exists a subsequence $\e_{k}\searrow 0$ such that $u_{\e_{k}}$ converges to a solution solving the original complex Monge-Amp\'{e}re flow equation \eqref{eq:cma-general} in the space of $C^{0}([0,T)\times M)\cap C^{\infty}((0,T)\times M)$ for any $T>0.$ And by the Laplacian estimate in Lemma \ref{lem-general-high} the flow solution $\dr(t)=\dr_{t}+\ddb u$ is complete as soon as $t>0.$ Moreover the uniqueness can also be proved using the same method in Proposition \ref{prop-unique}.\\

Now we begin to consider the normalized \ka-Ricci flow on $M:$
\begin{equation}\label{eq:general-krf2}
 \left\{ \begin{array}{rcl}
\dt\tilde{\dr}&=&-Ric(\tilde{\dr})-\tilde{\dr}\\
\tilde{\dr}(0)&=&\tilde{\dr}_{0} \end{array}\right.
\end{equation}
Similarly the solution to \eqref{eq:general-krf2} can be derived from the solution $\dr(t)$ of the unnormalized \ka-Ricci flow \eqref{eq:krf1} by the equation $$\tilde{\dr}(t)=e^{-t}\dr(e^{t}-1).$$ To describe the limit behavior of the normalized \ka-Ricci flow the remaining issue is to establish uniform estimates for
the solution $\tilde{\dr}(t).$ Now set $\tilde{\dr}_{t}:=e^{-t}\dr_{0}-(1-e^{-t})Ric(\dr_{M})$ and $\tilde{\dr}(t):=\tilde{\dr}_{t}+\ddb\tilde{u}$ we can derive the following complex Monge-Amp\'{e}re flow equation:
\begin{equation}\label{eq:general-cma2-nor}
\left\{ \begin{array}{rcl}
\dt\tilde{u}&=&\log\frac{(\tilde{\dr}_{t}+\ddb\tilde{u})^{n}}{\dr_{M}^{n}}-\tilde{u}\\
\tilde{u}(0)&=&0 \end{array}\right.
\end{equation}
Use the argument before we can prove that
there exists a unique solution $\tilde{u}$ in the space of $C^{0}([0,T)\times M)\cap C^{\infty}((0,T)\times M)$ for any $T>0.$ Moreover the metric
$\tilde{\dr}(t)$ is simultaneously complete as soon as $t>0.$ To derive uniform estimates for $\tilde{u}(t),$ instead of the direct application of the maximal principle in Lemma \ref{lem-C0-nor}, we will proceed as Theorem 4.1--5.1 in Lott-Zhang \cite{LZ} (here we want to thank M.C. Lee for pointing out this result).
\begin{lemma}\label{lem-general-C0-nor}
For $t_{0}>0,$ there exists a uniform constant $C>0$ such that
\begin{equation}\label{eq:C0-general-nor}
|\tilde{u}(t)|\leq C,\;|\dot{\tilde{u}}(t)|\leq C
\end{equation}
on $[t_{0},+\infty)\times M.$
\end{lemma}
\begin{proof}
It follows from \eqref{eq:general-cma2-nor} that 
\begin{align*}
(\dt-\Delta)\dot{\tilde{u}}&=-\dot{\tilde{u}}-e^{-t}tr_{\tilde{\dr}}(\dr_{0}+Ric(\dr_{M})),\\
(\dt-\Delta)\ddot{\tilde{u}}&=-\ddot{\tilde{u}}+e^{-t}tr_{\tilde{\dr}}(\dr_{0}+Ric(\dr_{M}))-\left|\frac{\partial^{2}\tilde{\dr}}{\partial t^{2}}\right|^{2}_{\tilde{\dr}},
\end{align*}
where $\tilde{\dr}$ denotes $\tilde{\dr}(t).$ It follows from these two equations that
\begin{equation}\label{eq:sum-dev}
(\dt-\Delta)(\dot{\tilde{u}}+\ddot{\tilde{u}})=
-\left|\frac{\partial^{2}\tilde{\dr}}{\partial t^{2}}\right|^{2}_{\tilde{\dr}}-(\dot{\tilde{u}}+\ddot{\tilde{u}}),
\end{equation}
which implies that 
$$(\dt-\Delta)e^{t}(\dot{\tilde{u}}+\ddot{\tilde{u}})\leq 0.$$
By the estimates for the solution to the unnormalized \ka-Ricci flow we know that for $t_{0}>0$ there exists a constant $C'>0$ such that 
$$|\tilde{u}(t_{0})|+|\dot{\tilde{u}}(t_{0})|\leq C'.$$ Thus for any $t>t_{0}$ it follows from the Omori-Yau's maximal principle that 
$$e^{t}(\dot{\tilde{u}}+\ddot{\tilde{u}})\leq C'e^{t_{0}}\leq C_{1}.$$ So we have 
\begin{equation}\label{eq:time-dev-decay}
e^{t}\dot{\tilde{u}}\leq C_{1}t+C_{2},
\end{equation}
 which implies that $$\tilde{u}(t)\leq C_{3},\;\dot{\tilde{u}}(t)\leq C_{3}.$$
For the lower bound estimates, considering that $Ric(\dr_{M})\leq-C\dr_{M},$ it follows from \eqref{eq:general-cma2-nor} that 
\begin{align*}
(\dt-\Delta)(\tilde{u}+\dot{\tilde{u}})&=-n-tr_{\tilde{\dr}}Ric(\dr_{M})\\
&\geq-n+Ctr_{\tilde{\dr}}\dr_{M}\geq-n+Cn\left(\frac{\dr^{n}_{M}}{\tilde{\dr}^{n}}\right)^{\frac{1}{n}}
=-n+Cne^{-\frac{\tilde{u}+\dot{\tilde{u}}}{n}}.
\end{align*}
Applying Omori-Yau maximal principle again and assuming the infimum of $\tilde{u}+\dot{\tilde{u}}$ is attained at $(x,t),$ at this point
it follows that $$-n+Cne^{-\frac{\tilde{u}+\dot{\tilde{u}}}{n}}\leq 0.$$ Thus $$\tilde{u}+\dot{\tilde{u}}\geq C_{4}.$$ Combine the upper bound estimates of
$\tilde{u}$ and $\dot{\tilde{u}}$ this lemma follows.
\end{proof}
Similar to Lemma \ref{lem-C2-nor} we have the following Laplacian estimate for $\tilde{u}:$
\begin{lemma}\label{lem-general-estimates}
There exists a uniform constant $C$ and $t_{0}>0$ such that for all $t>t_{0}$,
\begin{equation}\label{eq:C2-general-nor}
C^{-1}\dr_{M}\leq\tilde{\dr}_{t}+\ddb\tilde{u}\leq C\dr_{M}.
\end{equation}
\end{lemma}
By the estimates above we can derive uniform high order estimates from standard iterations. From \eqref{eq:time-dev-decay}
as time tends to infinity $\tilde{u}(t)$ is almost non-increasing. Considering the uniform $C^{0}$-estimate \eqref{eq:C0-general-nor}
we conclude that $\tilde{u}(t)$ converges to a unique limit $\tilde{u}_{\infty}$ and moreover this convergence is smooth by the high order estimates.
It follows from \eqref{eq:general-cma2-nor} that the limit $\tilde{u}_{\infty}$ satisfies the following equation:
\begin{equation}\label{eq:limit cma-general}
\log\frac{(\tilde{\dr}_{\infty}+\ddb\tilde{u}_{\infty})^{n}}{\dr_{M}^{n}}=\tilde{u}_{\infty},
\end{equation}
 which implies that
$$Ric(\tilde{\dr}_{\infty}+\ddb\tilde{u}_{\infty})=Ric(\dr_{M})-\ddb\tilde{u}_{\infty}.$$ Thus the limit of the solution to \eqref{eq:general-cma2-nor} is 
$\tilde{\dr}(\infty)=-Ric(\dr_{M})+\ddb\tilde{u}_{\infty}$ which satisfying $Ric(\tilde{\dr}(\infty))=-\tilde{\dr}(\infty).$ This completes the proof of Theorem \ref{mainthm2}. In some sense considering the existence of the complete \ka-Einstein metric on bounded domains of holomorphy by Cheng-Yau and Mok-Yau \cite{CY,MY}, we have the following corollary which can be thought as a generalization of Chau's work on bounded domains \cite{Chau}:
\begin{corollary}\label{coro-domain}
For any incomplete metric $\dr_{0}$ with uniformly bounded covariant derivatives there exists a family of solutions to the normalized \ka-Ricci flow
\eqref{eq:general-krf2} which is simultaneously complete as soon as $t>0$ and converges to the complete \ka-Einstein metric.
\end{corollary}

\section{Further discussions}

By Chau-Li-Tam's work \cite{CLT}, assume the initial metric $\dr_{0}$ is equivalent to the background metric $\dr_{M}$ in Theorem \ref{mainthm2}. If furthermore $\dr_{0}$ has bounded curvature or is the limit of a sequence of \ka\ metrics with bounded curvature, there exists a solution to the \ka-Ricci flow \eqref{eq:krf1} for infinite time. Considering our result Theorem \ref{mainthm2}, we conjecture that if the initial metric $\dr_{0}$ with bounded curvature or as the limit of a sequence of \ka\ metrics with bounded curvature, and there exists a constant $C>0$ such that $\dr_{0}\leq C\dr_{M}$ then we have
a solution for infinite time. More generally, for those results we still don't know whether we could drop the bounded curvature assumption, as Topping's work on surfaces \cite{Top1,Top2,Top3}. \\

Next, we still hope to extend Cheng-Yau's approximation method \cite{CY} in general pseudoconvex manifolds to the \ka-Ricci flow case. As we mentioned before the main difficulty is that unlike the \ka-Einstein case, we cannot use the local background Poincar\'{e} metric to control the flow solution from above. One possible idea is to apply Perelman's pseudolocality theorem to derive the local compactness of the \ka-Ricci flow solutions on the approximating domains. Also it will be wonderful if we can loose the assumptions on background metric, e.g., without the completeness or curvature bound. In
fact, we may consider general existence problem of the \ka-Ricci flow on Stein manifolds by this approach, especially when the initial metric has no curvature bound. In that general Stein manifold we only know that there exists a background metric with non-positive bisectional curvature descending from the standard embedding into the Euclidean space.\\

Finally we are curious that in what sense our construction of the flow solution is unique. Note that in Lee-Tam's work \cite{LT} they can construct complete solutions to Chern-Ricci flow in the same conformal class, which is closer to Topping's construction. It will be interesting to investigate the limit behavior of their approach. We also hope to know among all possible simultaneously complete flow solutions which solution is maximal stretched, as defined by Giesen-Topping \cite{GT2}, and whether it is unique.

\end{document}